\documentclass[11pt,reqno]{amsart}
\usepackage{amsmath,amssymb,amscd,amsthm,graphicx,enumerate,tikz-cd}
\usepackage{hyperref}
\hypersetup{breaklinks=true}
\usepackage{enumitem} 
\newlist{condenum}{enumerate}{1} 
\setlist[condenum]{label=\bfseries Condition \arabic*.,  ref=\arabic*, wide}
\usepackage[utf8]{inputenc}
\usepackage{amssymb}
\usepackage{amsthm}
\usepackage{mathrsfs}
\usepackage{amsfonts}
\usepackage{amsmath}
\usepackage{mathtools}
\usepackage{makecell}
\usepackage{xcolor}
\usepackage[margin=1.25in]{geometry}
\usepackage{lipsum}
\usepackage{soul}
\usepackage{caption}
\usepackage{cancel}
\usepackage{enumitem}
\definecolor{darkblue}{rgb}{0,0,0.3}
\definecolor{urlblue}{rgb}{0,0,0.7}
\hypersetup{
	colorlinks=true,
	citecolor=blue,
	linkcolor=darkblue,
	urlcolor=urlblue,
}

\numberwithin{equation}{section}
\theoremstyle{plain}
\newtheorem{theorem}{Theorem}[section]
\newtheorem{lem}[theorem]{Lemma}

\newtheorem{claim}{Claim}
\newtheorem{cor}[theorem]{Corollary}

\theoremstyle{definition}
\newtheorem{defn}[theorem]{Definition}
\newtheorem*{theorem*}{Theorem}

\newtheorem{rmk}[theorem]{Remark}

\newcommand{\R}{\mathbb{R}}
\newcommand{\RR}{\mathbb{R}^2}

\newcommand{\Rm}{\textnormal{Rm}}

\newcommand{\diam}{\textnormal{diam}}

\newcommand{\eps}{\varepsilon}
\renewcommand{\tilde}{\widetilde}

\newcommand{\D}{\nabla}
\newcommand{\p}{\partial}

\renewcommand{\P}[1]{{\big|\p #1\big|}}
\newcommand{\Ps}[1]{{|\p #1|}}
\DeclareMathOperator{\Lip}{Lip}
\DeclareMathOperator{\loc}{loc}
\DeclareMathOperator{\inj}{inj}
\renewcommand{\epsilon}{\varepsilon}
\let\div\undefined
\DeclareMathOperator{\div}{div}

\renewcommand{\bar}[1]{\overline{#1}}
\newcommand{\cC}{\mathcal{C}}
\newcommand{\Z}{\mathbb{Z}}

\newcommand{\TODO}[1]{}

\begin{document}

\title[3D manifolds with PSC and bounded geometry]{3-Manifolds with positive scalar curvature and bounded geometry}

\begin{abstract}
    We show that a complete contractible 3-manifold with positive scalar curvature and bounded geometry must be $\R^3$. We also show that an open handlebody of genus larger than 1 does not admit complete metrics with positive scalar curvature and bounded geometry. Our results rely on the maximal weak solution to inverse mean curvature flow due to the third-named author.
\end{abstract}

\author{Otis Chodosh}
\address{Otis Chodosh\hfill\break Department of Mathematics, Stanford University, CA 94305, USA}
\email{\href{mailto:ochodosh@stanford.edu}{ochodosh@stanford.edu}}

\author{Yi Lai}
\address{Yi Lai\hfill\break Department of Mathematics, UC Irvine, CA 92697, USA
}
\email{\href{mailto:ylai25@uci.edu}{ylai25@uci.edu}}

\author{Kai Xu}
\address{Kai Xu\hfill\break Department of Mathematics, Duke University, Durham, NC 27705, USA}
\email{\href{mailto:kai.xu631@duke.edu}{kai.xu631@duke.edu}}

\maketitle


\section{Introduction}

We are concerned here with the classification problem for $3$-manifolds $M$ admitting complete Riemannian metrics with positive scalar curvature; see Question 27 in Yau's list \cite{yau1982problem}. A fundamental observation of Schoen--Yau relates scalar curvature to the stability of minimal surfaces \cite{SchoenYau1979}, which ultimately leads to several topological obstructions to the existence of complete nonnegative scalar curvature metrics on a $3$-manifold $M$. First, if $\pi_1(M)$ contains a surface subgroup, then $M$ is flat by Schoen-Yau \cite[Theorem 4]{schoen1982complete} and Gromov--Lawson \cite[Theorem 8.4]{gromovlawson1983}. Second, J.\ Wang \cite{Wang1,Wang2} proved that if $M$ is contractible and admits an exhaustion by solid tori, then $M \cong \R^3$ (cf.\ Lemma \ref{lem:contractT2}). In particular, the Whitehead manifold does not admit such a metric. See also relevant results \cite{CRZ23,zhu2021,Chen:SYS,ChCh24,lott2024obstructions,Zhu:calabi}. In spite of these results, the general classification for noncompact $M$ admitting such a metric is widely open. In particular, we note the following special cases:
\begin{itemize}
    \item If $M$ is a contractible $3$-manifold and admits a complete metric of nonnegative scalar curvature, do we have $M \cong \R^3$? (Asked by J.\ Wang \cite{Wang1,Wang2}, cf.\ \cite{changweinbergeryu2010}.)
    \item If $M_\gamma $ is an open handlebody of genus $\gamma$, and admits a complete metric of non-negative scalar curvature, do we have $\gamma \leq 1$? (Asked by Gromov \cite[\S 3.10.2]{gromov4lectures}.)
\end{itemize}

In this paper, we resolve these two questions under the \emph{additional} assumption that the metric has bounded geometry:
\begin{equation}\label{eq:bounded_geometry}
    |\Rm|\le\Lambda,\qquad \textnormal{inj}\ge \Lambda^{-1},\tag{BG}
\end{equation}
by combining minimal surface obstructions with new topological constraints obtained using inverse mean curvature flow. In particular, our main results are as follows. Hereafter, we use $R$ to denote the scalar curvature.

\begin{theorem}\label{thm-main:R3}
    Let $(M,g)$ be a complete, connected, contractible Riemannian 3-manifold satisfying $R\geq0$ and \eqref{eq:bounded_geometry}. Then $M$ is diffeomorphic to $\mathbb{R}^3$.
\end{theorem}

\begin{theorem}\label{thm-main:handlebody}
    Let $M_\gamma$ denote the interior of the handlebody of genus $\gamma$. If $(M_\gamma,g)$ is a complete Riemannian $3$-manifold satisfying $R\geq0$ and \eqref{eq:bounded_geometry}, then $\gamma \leq 1$.
\end{theorem}

Note that $\R^3$ and $\R^2\times\mathbb S^1$ (corresponding to $\gamma=0,1$ in Theorem \ref{thm-main:handlebody}) both admit complete metrics with $R\geq0$ and bounded geometry. Concrete examples are a capped-off half-cylinder (which actually has $R\geq 1$) and the product of Cigar soliton and $\mathbb S^1$, respectively.


For the stronger uniformly positive scalar curvature condition $R\geq1$, J.\ Wang has obtained a complete classification \cite{Wang23positive}: these 3-manifolds are infinite connect sums of spherical space forms and $S^2\times S^1$. In particular, the only contractible manifold or handlebody admitting such a metric is $ \R^3$. We note that earlier work of Bessi\`eres--Besson--Maillot \cite{Besson} used Ricci flow to prove such a classification with an additional bounded geometry assumption. 

A crucial tool in J.\ Wang's classification \cite{Wang23positive} is the $\mu$-bubbles introduced by Gromov \cite[\S 3.7.2]{gromov4lectures}. From this one obtains that a $3$-manifold $M$ with $R\geq1$ admits an exhaustion by regions whose boundary consists of $\mathbb S^2$'s. This places very strong topological constraints on $M$ (see also \cite{chodosh2024completeriemannian4manifoldsuniformly}). When the scalar curvature is only non-negative, it's not known how to use the $\mu$-bubble method to obtain such an exhaustion (in particular, the genus one handlebody shows that one must allow for torus boundaries in this exhaustion). As we will discuss below, one of our main contributions here is to find an exhaustion with only sphere or tori boundary in certain $3$-manifolds with non-negative scalar curvature.

\subsection{Topological obstructions via inverse mean curvature flow}

The key novelty introduced in this paper is the use of inverse mean curvature flow as a replacement for $\mu$-bubbles in topological applications. A family of hypersurfaces is a smooth inverse mean curvature flow (IMCF) if it evolves in the outwards pointing direction with speed $\frac{1}{H}$, where $H$ denotes the mean curvature.

The relevance of this flow to scalar curvature is the following: if $(M,g)$ is a Riemannian 3-manifold with nonnegative scalar curvature, and $\Sigma_t\subset M$ is a compact family evolving by the smooth IMCF, then $|\Sigma_t| = e^t |\Sigma_0|$ and
\begin{equation}\label{eq-intro:Geroch}
    \frac{d}{dt} \int_{\Sigma_t} H^2 \leq -\frac 12 \int_{\Sigma_t} H^2 + 4\pi  \chi(\Sigma_t).
\end{equation}
This is known as the Geroch monotonicity formula \cite{geroch1973energy}. In particular, if the flow exists for all time $t \in [0,\infty)$ then $\Sigma_t$ cannot have genus $\geq 2$ for all large $t$, since otherwise \eqref{eq-intro:Geroch} would force $ \int_{\Sigma_t} H^2$ to be negative for $t \gg1$, which is impossible. In particular, this implies that $M$ admits an exhaustion by regions with sphere or torus boundaries, strongly constraining its topology.

However, there are major issues with the assumption of long-time existence in practice. First, singularities are likely to develop along the flow. Secondly, it's possible that the flow rushes to infinity in finite time if the infinity is not ``large'' enough.

To allow for singularities, we can use the notion of weak IMCF introduced by Huisken--Ilmanen \cite{HI01} en route to their proof of the Riemannian Penrose inequality. Intuitively, this solution can be described as running the smooth flow except at each time replacing $\Sigma_t$ by its least area enclosure. As proven by Huisken--Ilmanen, the Geroch monotonicity \eqref{eq-intro:Geroch} remains true for weak solutions as long as they exist.

A weak IMCF that does not rush to infinity in finite time is called proper. In \cite{HI01} and \cite{BrayNeves04,BrendleChodosh,BrendleHungWang,Shi:iso,Chodosh:AH-iso,ChodoshEichmairShiZhu,JaureguiLee,CESY,ChodoshEichmair:anm,AFMM,huisken2024inverse}, proper IMCF are used as central tools in several scalar curvature problems. We also refer to \cite{kotschwarNi,mariRigoliSetti,Xu24proper} about proper IMCF.

In our current setting, we inevitably encounter non-proper weak IMCFs (i.e. weak IMCFs that rush to infinity within finite time). We make essential use of the third-named author's recent work \cite{Xu24obstacle}, which shows that $(M,g)$ always admits a ``maximal'' (or ``innermost'', ``slowest'') weak IMCF.
Assuming bounded geometry and one-endedness of $M$, we show that the maximal weak solution satisfies exactly one of the following three possibilities, see Lemma \ref{l: three cases}: 

\begin{enumerate}
\item [(i)] Proper: The solution exists and remains bounded for all time.
\item [(ii)] Sweeping: The solution entirely moves to infinity at some time $T\in(0,\infty)$.
\item [(iii)] Escaping: The solution exists until a time $T\in(0,\infty)$, then ``jumps'' to infinity. 
\end{enumerate}
In the proper case (i), we can obtain a topological obstruction using the monotonicity formula \eqref{eq-intro:Geroch} as above. Now we consider the remaining cases (ii) (iii); see Figure \ref{fig:sweep_escape} for examples of each of these cases.

\vspace{3pt}

\begin{figure}[h]
    \centering
    \includegraphics[width=0.9\textwidth]{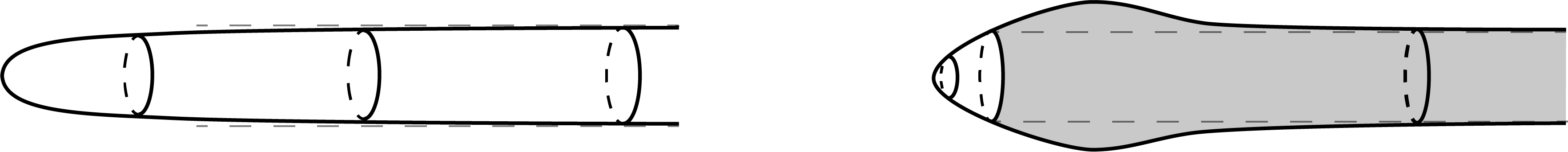}
    \caption{An IMCF sweeping out the manifold at $t=T$ (left) and one that escapes at $t=T$ (right).}
    \begin{picture}(0,0)(0,-13)
        \put(-166,70){$\Sigma_0$}
        \put(-120,70){$\Sigma_{T-0.1}$}
        \put(-54,70){$\Sigma_{T-0.01}$}
        \put(27,61){$\Sigma_0$}
        \put(42,70){$\Sigma_T$}
        \put(80,50){$u\equiv T$}
    \end{picture}\label{fig:sweep_escape}
\end{figure}
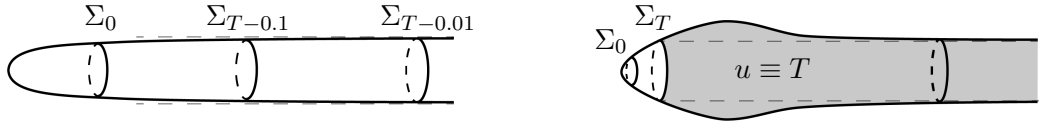

First, we consider the case of sweeping flow. We can show that for a sequence of times $t_i\nearrow T$, the surfaces $\Sigma_{t_i}$ have uniformly bounded diameters, are uniformly $C^{1,\alpha}$-smooth, and are ``almost" area-minimizing. Taking a subsequential limit, we obtain an area-minimizing hypersurface in some limit of $M$ at infinity. By the scalar curvature lower bound, this limiting hypersurface must be $\mathbb S^2$ or $\mathbb T^2$, which in turn implies that all but finitely many $\Sigma_{t_i}$ are $\mathbb S^2$ or $\mathbb T^2$. This again puts strong constraints on the topology of $M$.

Finally, we consider the case of escaping flow. In order to find a nice exhausting sequence and perform a limiting argument, we make a small perturbation of the metric so that it becomes ``larger at infinity''. This will delay the escape time of the maximal IMCF, thus some new level set will form in the edited region. Letting the edited region diverge to infinity and making the perturbation smaller and smaller, we obtain another diverging sequence of hypersurfaces which are ``almost" area-minimizing as well. Then the limiting argument in case (ii) is employed to prove the main theorems. We refer to Section \ref{sec:escaping} for more details. 


\begin{figure}[h]
    \centering
    \includegraphics[width=0.6\linewidth]{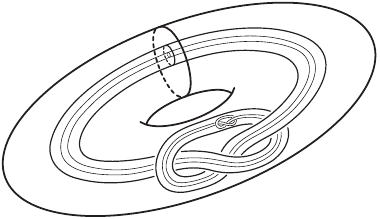}
    \caption{Iterated trefoils. The figure shows the embedding $\Omega_1\subset\Omega_2\subset\Omega_3$.}
    \label{fig:trefoil}
\end{figure} 

\subsection{Manifolds that are nested by solid tori} 
Our work suggests that it might be possible to give a complete classification of $3$-manifolds admitting a metric of non-negative scalar curvature with bounded geometry. The techniques developed here suggest that a key step in this classification (which we have not been able to resolve) would be the classification of $3$-manifolds that admit such metrics and also admit an exhaustion $\Omega_1\subset\Omega_2\subset\dots$, $M = \cup_{i=1}^\infty \Omega_i$, with each $\Omega_i$ a solid torus. As a concrete example, suppose that $\Omega_i$ embeds in $\Omega_{i+1}$ as a trefoil knot (see Figure \ref{fig:trefoil}); we do not know if $M$ admits a complete metric of non-negative scalar curvature and bounded geometry.

\subsection{Acknowledgements}  
O.C.~was supported by a Terman Fellowship and an NSF
grant (DMS-2304432). Y.L.~is supported by NSF grant DMS-2203310. 
The authors wish to thank Richard Bamler, Hubert Bray, Marcus Khuri, Chao Li, John Lott for their interest. Finally, we thank the hospitality of Simons Laufer Mathematical Institute, where a major part of this work was completed during the authors' visits.

\section{Topological preliminaries}

In the following, we assume that all manifolds are noncompact without boundary.

\begin{lem}\label{lem:contract-ends}
Assume $M$ is contractible and $\dim M \geq 2$. Then $M$ has only one end.
\end{lem}
\begin{proof}
    If $M$ has more than one end, then we can find a smooth domain $K\Subset M$ such that $M\setminus K$ has more than one unbounded components. Let $E$ be an unbounded component of $M\setminus K$, then $M\setminus \bar E$ is also unbounded, hence $[\partial E] \neq 0 \in H_{n-1}(M;\mathbb{Z})$. This is impossible since $M$ is contractible.
\end{proof}

\begin{lem}\label{lem:contract-irred}
    Assume $M$ is a contractible 3-manifold. Then $M$ is irreducible.  
\end{lem}
\begin{proof}
    A smoothly embedded $2$-sphere in $M$ must bound a precompact region $\Omega$. By Van Kampen's theorem we have $\pi_1(M)=\pi_1(\Omega)*\pi_1(M\setminus\Omega)=1$, and thus $\pi_1(\Omega)=1$. So the Poincar\'e conjecture \cite{Pel1,Pel2} implies  that $\Omega$ is a ball.
\end{proof}

In Lemma \ref{lem:contractS2}, \ref{lem:contractT2} below, we suppose that $M$ is a contractible 3-manifold, and that $\Omega_1\Subset\Omega_2\Subset\dots\Subset M$ with $\bigcup_{i}\Omega_i = M$ is an exhaustion of $M$, with each $\p\Omega_i$ being smooth and connected.

\begin{lem}\label{lem:contractS2}
    If $\partial\Omega_i\cong \mathbb S^2$ for all $i=1,2,\dots$, then $M\cong\mathbb{R}^3$.
\end{lem}
\begin{proof}
    As in Lemma \ref{lem:contract-irred}, each $\Omega_i$ is a $3$-ball. Thus $M$ is an increasing union of smoothly embedded 3-balls. Thus, $M$ is diffeomorphic to $\R^3$. This final step follows from \cite[Theorem 1]{HP} but since the present situation is much simpler, we sketch the proof here. Capping off $\Omega_{i+1}$ with a $3$-ball we can use Alexander's theorem to see that $\overline{\Omega_{i+1}}\setminus\overline{\Omega_i} \cong S^2 \times [0,1]$. We can assemble these diffeomorphisms together to show that $M\setminus \Omega_1 \cong S^2 \times [1,\infty)$. This completes the proof.  
\end{proof}

The next result essentially follows from the work of J.\ Wang \cite{Wang1} (see also \cite{Wang2}), but we need a slightly stronger result than the one stated there. 
\begin{lem}\label{lem:contractT2}
    If $\partial\Omega_i \cong \mathbb T^2$ for all $i=1,2,\dots$, and $M$ admits a complete metric of positive scalar curvature, then $M\cong \mathbb{R}^3$. 
\end{lem}
\begin{proof}
    Since $\pi_1(M)=1$, the map $\pi_1(\p\Omega_i)\to\pi_1(M)$ is not injective, and thus the loop theorem implies that $\p\Omega_i$ is compressible.
    
    We claim we can find a compressing disk $D$ of $\p\Omega_i$ such that $D\subset \Omega_i$ or $D\subset M\setminus\Omega_i$. Let $D$ be a compressing disk. We may assume $D$ is transverse to $\p\Omega_i$, thus $D\cap\p\Omega_i$ consists of finitely many disjoint circles. Let $\gamma$ be an innermost circle in the intersection. We may assume that $\gamma$ is homotopically nontrivial in $\p\Omega_i$: otherwise, it bounds a disk $D_1\subset \p\Omega_i$ and also a disk $D_2\subset \Omega_i$ or $D_2\subset M\setminus\Omega_i$, thus we can replace $D_2$ by $D_1$ and push it slightly outwards or inwards to remove $\gamma$ from the intersection. Through this reduction process, we obtain a homotopically non-trivial circle $\gamma\subset\p\Omega_i$, with a compressing disk $D'\subset\Omega_i$ or $D'\subset M\setminus\Omega_i$. The disk $D'$ proves our claim.

    If $D\subset M\setminus\Omega_i$ for some $i$, then by doing surgery along $D$, we can find another $\Omega_i'\supset\Omega_i$ with $\partial\Omega_i'\cong \mathbb S^2$. If this occurs for infinitely many $i$, we can then pass to a subsequence so that Lemma \ref{lem:contractS2} applies. As such, it suffices to assume that $D\subset \Omega_i$ for all $i$. Then the surgery of $\p\Omega_i$ along $D$ is a sphere and thus (by Lemma \ref{lem:contract-irred}) bounds a ball $B$ in $M$. This implies that $\Omega_i$ is a solid torus. 

    Suppose that for infinitely many $i$, the torus $\partial\Omega_{i+1}$ is compressible in $\Omega_{i+1}\setminus \Omega_i$. Then, by surgery along the compressing disk, we can again find an exhaustion by regions with spherical boundary, and thus Lemma \ref{lem:contractS2} applies. As such, we can assume that $\p\Omega_{i+1}$ is not compressible in $\Omega_{i+1}\setminus\Omega_i$ for all $i$. This implies that $M$ is a contractible genus-one manifold in the sense of \cite[Definition 2.14]{Wang1}. Such an $M$ cannot admit a complete metric of positive scalar curvature by \cite[Theorem 1.2]{Wang1}. This completes the proof.
\end{proof}

The following two lemmas concern the connectedness of boundaries.

\begin{lem}\label{lem:connectedness}
    Suppose $M$ is one-ended, and for all $K\Subset M$ the map $H_1(M\setminus K,\mathbb Z)\to H_1(M,\Z)$ is surjective. If $E\Subset M$ is a connected $C^1$ domain such that $M\setminus E$ does not have bounded connected components, then $\p E$ is connected.
\end{lem}
\begin{proof}
    Since $M$ is one-ended, $M\setminus E$ must be connected. If $\p E$ is disconnected, then there is a component of $\p E$ that has nonzero algebraic intersection with a closed loop $\gamma$. However, by the lemma's condition, $\gamma$ is homologous to a loop that is disjoint from $\p E$. This leads to a contradiction.
\end{proof}

\begin{lem}\label{lem:connectedness2}
    Suppose $M$ is one-ended with $b_1(M)<\infty$. Then there is a compact set $S$ with the following property: if $E\Subset M$ is a connected $C^1$ domain such that $M\setminus E$ does not have bounded connected components, then
    \begin{enumerate}[label={(\roman*)}, nosep]
        \item either $\p E$ is connected,
        \item or every connected component of $\p E$ intersects $S$.
    \end{enumerate}
\end{lem}
\begin{proof}
    Let $\{\gamma_i\}$ be a finite collection of loops that generates $H_1(M,\mathbb Z)$ modulo torsion. Choose any compact set $S\Supset\bigcup\gamma_i$. We claim that $S$ satisfies the lemma's condition. Let $E$ be as stated in the lemma. Since $M$ is one-ended, $M\setminus E$ must be connected and noncompact. Also, recall that $E$ is connected. So if $\p E$ is disconnected, each connected component of $\p E$ must be non-separating. Let $\Sigma$ be any component of $\p E$; thus $\Sigma$ has nonzero algebraic intersection with some loop $\sigma$. Since $\sigma$ is homologous to a linear combination of $\gamma_i$ modulo torsion, and $\gamma_i\subset S$ have zero intersection numbers with $\Sigma$, it follows that the $\sigma$ has zero intersection numbers with $\Sigma$, which is a contradiction.
\end{proof}

\begin{lem}\label{lemma:handlebody_exhaustion}
    Let $M^3$ be an open handlebody. Let $\Omega_1\Subset\Omega_2\Subset\cdots\Subset M$ be a $C^1$ exhaustion of $M$, such that each $\Omega_i$ and $M\setminus\Omega_i$ is connected. Then $\textnormal{genus}(\p \Omega_i)\ge \textnormal{genus}(M)$ for all large $i$.
\end{lem}
\begin{proof}
    We choose $U\cong \Sigma \times [0,1]$ to be a collar neighborhood of $\partial M \cong \Sigma \times \{1\}$. For $i$ sufficiently large, it holds $\Sigma\times \{0\}\subset \Omega_i \cap U$. For all large $i$ we have
    \begin{enumerate}[label={(\arabic*)}, nosep]
        \item $\p\Omega_i$ is connected (by Lemma \ref{lem:connectedness}),
        \item $\p\Omega_i$ separates $\Sigma\times \{0\}$ and $\Sigma\times \{1\}$.
    \end{enumerate}
    \noindent Hence, the natural projection map $\p\Omega_i\to\p M$ has nonzero degree, and by Kneser's theorem (cf.\ \cite{ryabichev}), we have $\text{genus}(\p\Omega_i)\geq\text{genus}(\p M)$.
\end{proof}

\begin{lem}\label{lemma:one_end_cpt}
    Assume $M$ has only one end, $E\subset M$ is connected, with $E\ne M$, $\p E$ compact, and $M\setminus E$ unbounded. Then $E\Subset M$.
\end{lem}
\begin{proof}
    Since $M$ is one-ended, there is a connected compact set $K\Supset\p E$ such that $M\setminus K$ is connected and unbounded. Since $E$ is connected and $\p E\cap(M\setminus K)=\emptyset$, we have either $E\cap(M\setminus K)=\emptyset$ or $E\supset M\setminus K$. The first case implies $E\subset K$ hence proves the lemma. The second case implies $M\setminus E\subset K$, which is excluded by the hypotheses of the lemma.
\end{proof}

\section{IMCF Preliminaries}\label{sec:IMCF}

We provide a brief introduction to the weak IMCF. Here we only state results needed in the subsequent sections; we refer the reader to \cite{HI01, Xu24obstacle} for much detailed introduction to this subject. We fix the following notations:
\begin{enumerate}[nosep]
    \item $M$ will always denote a complete, connected, noncompact Riemannian manifold, with dimension $n\leq 7$ (so that minimal surface singularities do not appear).
    \item For a function $u$, we denote $E_t(u)=\{u<t\}$ and $E_t^+(u)=\{u\leq t\}$. When there is no ambiguity, we will write $E_t,E_t^+$ for brevity.
    \item When $u$ is defined in some domain $\Omega$, we view $E_t$ as a subset of $\Omega$, hence a subset of $M$. The perimeter $\Ps{E_t}$ is viewed as the perimeter in $M$, so it contains $\Ps{E_t\cap\Omega}$ as well as $\Ps{E_t\cap\p\Omega}$.
    \item We will use $C(\cdots)$ to denote generic constants which depend on the items in the parentheses. The constant may change from line to line.
\end{enumerate}

A \textit{smooth IMCF} is defined as a smooth function $u$ with non-vanishing gradient, satisfying the equation
\begin{equation}\label{eq-prelim:smooth_IMCF}
    \div\Big(\frac{\D u}{|\D u|}\Big)=|\D u|.
\end{equation}
Note that this is a level set flow: setting $\Sigma_t=\{u=t\}$, one finds that $\{\Sigma_t\}$ is a family of hypersurfaces evolving at the speed of $1/H_t$, where $H_t$ is the mean curvature of $\Sigma_t$. To see this, it can be calculated that $H_t=\div\big(\frac{\D u}{|\D u|}\big)$ while the speed of evolution equals $\frac1{|\D u|}$.

The weak IMCF is defined as a variational form of \eqref{eq-prelim:smooth_IMCF}.

\begin{defn}[weak IMCF and its energy functional]\label{def:weak_sol} {\ }

    Given a domain $\Omega\subset M$. For a function $u\in\Lip_{\loc}(\Omega)$, a set $E$ with locally finite perimeter, and a domain $K\Subset\Omega$, we define the energy functional as
    \[J_u^K(E)=|\p^*E\cap K|-\int_{E\cap K}|\D u|.\]
    where $\p^*E$ is the reduced boundary of $E$ (see \cite{Maggi}).
    
    We say that $u$ is \textit{a weak solution of IMCF in $\Omega$}, if $u\in\Lip_{\loc}(\Omega)$, and each level set $E_t=\{u<t\}$ locally minimizes $J_u$ in the following sense: for any $K\Subset\Omega$ and any domain $F$ with $F\Delta E_t\Subset K$, it holds $J_u^K(E_t)\leq J_u^K(F)$.
\end{defn}

Note that that a smooth IMCF is also a weak IMCF \cite[Lemma 2.3]{HI01}. Let us include a brief proof here: suppose $u$ solves \eqref{eq-prelim:smooth_IMCF} in $\Omega$, $t\in\RR$, $K\Subset\Omega$ and $F$ is a competitor as in Definition \ref{def:weak_sol}. Then
\begin{align}
    J_u^K(F)-J_u^K(E_t) &= \big|\p^*F\cap K\big|-\big|\p^*E_t\cap K\big|-\int(\chi_F-\chi_{E_t})|\D u| \nonumber\\
    &\geq \int_{\p^*F\cap K}\nu_F\cdot\frac{\D u}{|\D u|}-\int_{\p^*E_t\cap K}\nu_{E_t}\cdot\frac{\D u}{|\D u|}-\int(\chi_F-\chi_{E_t})|\D u|, \label{eq:energy_comp}
\end{align}
where we used $\nu_F\cdot\frac{\D u}{|\D u|}\leq1$ and $\nu_{E_t}\cdot\frac{\D u}{|\D u|}=1$. Then applying the divergence theorem and using the IMCF equation \eqref{eq-prelim:smooth_IMCF}, the right side of \eqref{eq:energy_comp} equals zero.

Next, we define the initial value problem for the weak IMCF.

\begin{defn}[initial value problem]\label{def:ivp}
    Given a domain $\Omega\subset M$ and a $C^{1,1}$ domain $E_0\Subset\Omega$. We say that $u$ is a weak solution of IMCF in $\Omega$ with initial value $E_0$, if

    (i) $u\in\Lip_{\loc}(\Omega)$ and $E_0=\{u<0\}$;

    (ii) $u|_{\Omega\setminus\overline{E_0}}$ is a weak solution of IMCF in $\Omega\setminus\overline{E_0}$.
\end{defn}

Note that the specific value of $u$ inside $E_0$ is not important, and a weak IMCF with initial value $E_0$ is not necessarily a weak IMCF in $\Omega$.

\begin{lem}[immediate properties]\label{lemma:ivp_properties}
    For $\Omega,E_0,u$ as in Definition \ref{def:ivp}, we have:

    (i) each $E_t$ ($t>0$) is a $C^{1,\alpha}$ hypersurface;
    
    (ii) each $E_t$ ($t>0$) is locally outward perimeter-minimizing\footnote{We say that a set $E$ is locally outward perimeter-minimizing in a domain $\Omega$, if for all $K\Subset\Omega$ and $F\supset E$ with $F\setminus E\Subset K$, it holds $|\p^* E\cap K|\leq|\p^* F\cap K|$.} in $\Omega$ (in particular, $\Omega\setminus E$ has no precompact connected components);

    (iii) $|\p E_t|=e^{t-s}|\p E_s|\leq e^t|\p E_0|$ for all $0<s<t$, as long as $E_t\Subset\Omega$. If further $E_0$ is locally outward minimizing in $\Omega$, then $|\p E_t|=e^t|\p E_0|$.
\end{lem}
\begin{proof}
    (i) follows from the standard regularity results in geometric measure theory, see \cite[p.13]{Xu24obstacle}. For (ii) see \cite[Lemma 2.6]{Xu24obstacle}. (iii) follows from \cite[Lemma 1.4, 1.6]{HI01}.
\end{proof}

The following maximum principle is useful in later arguments.

\begin{lem}[semilocal uniqueness, {\cite[Theorem 2.2]{HI01}}]\label{lemma:max_principle} {\ }

    Let $E_0\Subset\Omega$, and $u,v$ be weak solutions of IMCF in $\Omega$ with initial value $E_0$. If for some $t>0$ we have $E_t(u)\Subset\Omega$ and $E_t(v)\Subset\Omega$, then $E_t(u)=E_t(v)$, and $u=v$ in $E_t(u)\setminus E_0$.
\end{lem}

The following types of solutions will occur in the main proof:

\begin{defn}[solution types]\label{def:sol_types} {\ }

    Let $u$ be a weak IMCF in $M$ with an initial value $E_0$, as in Definition \ref{def:ivp}. We say that:

    (i) $u$ is \textit{proper}, if  $\lim_{x\to\infty}u(x)=+\infty$, or equivalently, if $E_t\Subset M$ for all $t\in[0,\infty)$; 

    (ii) $u$ is \textit{sweeping}, if $T:=\sup(u)\in(0,\infty)$, and $E_t\Subset M$ for all $t<T$, and $E_T=M$.

    (iii) $u$ is \textit{instantly escaping}, if $T:=\sup(u)\in(0,\infty)$, and $E_T\Subset M$, and $u\equiv T$ in $M\setminus E_T$. In case (iii), we call $T$ the escape time of $u$.
\end{defn}

See Figure \ref{fig:sweep_escape} for examples of sweeping and escaping flows. Also, see Figure \ref{fig:infinite_bubbles} below for a complicated case of sweeping maximal IMCF, where there are infinitely many bubble-shapes hence infinitely many jumps in the flow.

\begin{figure}[h]
    \captionsetup{width=.9\linewidth}
    \centering
    \includegraphics[width=0.6\linewidth]{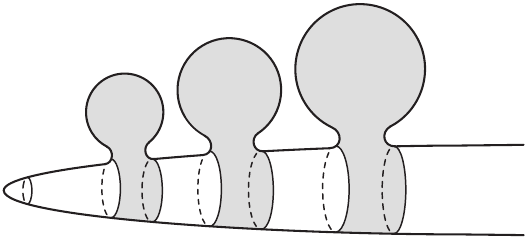}
    \caption{A complicated case of sweeping IMCF (grey regions represent jumps)}\label{fig:infinite_bubbles}
    \begin{picture}(0,0)
        \put(-140,65){$E_0$}
        \put(97,115){$\cdots\cdots$}
    \end{picture}
\end{figure}
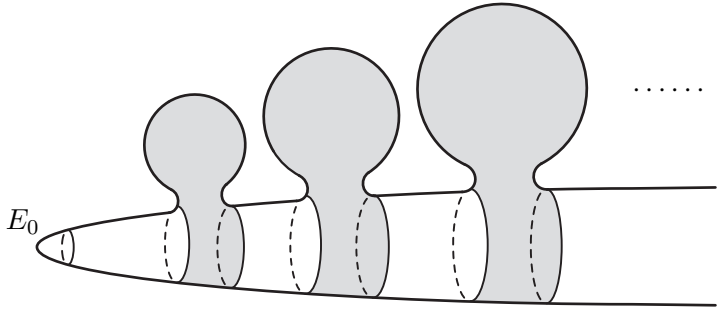

Regarding proper IMCF, we have the well-known Geroch monotonicity formula:

\begin{lem}[Monotonicity]\label{lem:monotonicity} {\ }

    Let $M$ be 3-dimensional and has nonnegative scalar curvature. Let $u$ be a proper weak IMCF in $M$ with a $C^{1,1}$ initial value $E_0\Subset M$. Then it holds (in the Gronwall sense)
    \[\frac d{dt}\int_{\p E_t}H^2\leq  4\pi\chi(\p E_t)-\frac12\int_{\p E_t}H^2.\]
\end{lem}
\begin{proof}
    This follows from \cite[Formula (5.22)]{HI01}; see also \cite[p. 395--396]{HI01} for computations in the smooth case.
\end{proof}

In the case that no proper solutions exist, we will make use of maximal weak solutions. A weak solution $u$ satisfying Definition \ref{def:ivp} is called \textit{maximal}, if $u\geq v$ for all weak solutions $v$ having the same initial value as $u$. We refer to \cite{Xu24obstacle} for a systematic study of this object. In the following lemma, we summarize its properties that will be used later. Note that a proper IMCF is also a maximal IMCF, hence the results below are applicable to proper IMCF as well.

\begin{lem}[properties and approximation of maximal solutions]\label{lemma:maximal_sol} {\ }

    \begin{enumerate}[label={(\roman*)}, itemsep=0.5ex]
        \item For any $C^{1,1}$ initial value $E_0\Subset M$, there exists a unique maximal solution $u$ of IMCF in $M$ with initial value $E_0$. Such solution has the following properties:
        \begin{enumerate}[itemsep=0.3ex]
            \item if $M$ satisfies \eqref{eq:bounded_geometry}, then we have the gradient estimate
            \[|\D u|\leq C(\Lambda,H_{\p E_0}^+)\qquad\text{outside $E_0$},\]
            where $H_{\p E_0}^+$ is the maximum of the mean curvature of $\p E_0$;
            \item if $E_0$ is connected, then $E_t$ is connected for each $t>0$;
            \item for each $t>0$, $M\setminus E_t$ does not have bounded connected components;
            \item $\Ps{E_t}\leq e^t\Ps{E_0}$ for all $t>0$.
        \end{enumerate}

        \item For any sequence of smooth domains $E_0\Subset\Omega_1\Subset\Omega_2\Subset\cdots\Subset M$ with $\bigcup_{l=1}^\infty\Omega_l=M$, there are functions $u_l\in\Lip_{\loc}(\Omega_l)$ such that:
        \begin{enumerate}[itemsep=0.3ex]
            \item each $u_l$ is the maximal weak solution of IMCF on $\Omega_l$ with initial value $E_0$;
            \item each $\p E_t(u_l)$ is a $C^{1,\alpha}$ surface in $M$;
            \item $u_1\geq u_2\geq u_3\geq\cdots$ outside $E_0$, and $\lim_{l\to\infty}u_l=u$ in $C^0_{\loc}(M\setminus E_0)$, where $u$ is the solution given in (i);
            \item for each $l$ and $t>s\geq0$, we have $\Ps{E_t(u_l)}\leq e^{t-s}\Ps{E_s(u_l)}$.
            \item if $E_t(u)\Subset\Omega_l$ for some $t>0$ and $l\geq1$, then $E_t(u_l)=E_t(u)$, and we have $u_l=u$ in $E_t(u_l)$.
        \end{enumerate}
    \end{enumerate}
\end{lem}
\begin{proof}
    For item (ii), we set each $u_l$ to be the weak IMCF with initial value $E_0$ and outer obstacle $\p\Omega_l$, given by \cite[Theorem 6.1]{Xu24obstacle}. Items (iia,\,iib) follow from items (iii,\,iv) of \cite[Theorem 6.1]{Xu24obstacle}. Item (iic) and the existence claim in item (i) follows from \cite[Theorem 7.1]{Xu24obstacle} and its proof. Item (iid) follows from \cite[Corollary 3.14]{Xu24obstacle}. Item (iie) is proved as follows: by item (iic) we have $E_t(u_l)\subset E_t(u)\Subset\Omega_l$, then the conclusion is implied by Lemma \ref{lemma:max_principle} (interior maximum principle). Items (ia,\,ib,\,id) follow directly from \cite[Theorem 7.2]{Xu24obstacle}. Item (ic) follows from Lemma \ref{lemma:ivp_properties}(ii).
\end{proof}
   
    


\begin{lem}[regularity and density bounds]\label{lemma:reg_density} {\ }

    Let $M$ satisfy \eqref{eq:bounded_geometry}, and $u$ be a weak solution of IMCF in some domain $\Omega\subset M$. Moreover, assume the gradient bound $|\D u|\leq L$ on $\Omega$. Then for all $t\in\R$ the following hold:
    
    (i) for all $x\in\p E_t(u)$ and $r\leq\Lambda^{-1}$ with $B(x,2r)\subset\Omega$, we have the density bound
    \[\P{E_t\cap B(x,r)}\geq C(L,\Lambda)r^{n-1};\]
    
    (ii) if $M$ further satisfies $|\D^k\Rm|\leq\Lambda'$, $\forall\,1\leq k\leq 5$, then $\p E_t(u)$ is a $C^{1,\alpha}$ surface, where the $C^{1,\alpha}$ norm\footnote{We say that the $C^{1,\alpha}$ norm of a hypersurface $\Sigma$ is bounded by $C$ near $x$, if there exists a geodesic normal coordinates in $B(x,C^{-1})$ in which $\Sigma$ is the graph of a function $f$ with $\|f\|_{C^{1,\alpha}}\leq C$.} near $x\in\p E_t(u)$ depends only on $\alpha,L,\Lambda,\Lambda'$, and $d(x,\p\Omega)$.
\end{lem}
\begin{proof}
    Note that the gradient bound and Definition \ref{def:weak_sol} implies
    \[\Ps{E_t\cap K}\leq\Ps{F\cap K}+L|E_t\Delta F|,\quad\text{for all $F$ with $E_t\Delta F\Subset K\Subset\Omega$.}\]
    Then item (i) follows from the classical area monotonicity formula, while (ii) follows from the standard regularity theory, see \cite[Chapter 21]{Maggi}.
\end{proof}

\begin{cor}[diameter bounds for level sets]\label{cor:diam_bound} {\ }

    Let $E_0\Subset M$ be a $C^{1,1}$ domain, and $u$ be the maximal solution of IMCF in $M$ with initial value $E_0$, as given by Lemma \ref{lemma:maximal_sol}(i). Denote $\cC(\p E_0):=\max\big\{\diam(\p E_0),|\p E_0|,H_{\p E_0}^+\big\}$.

    (i) Suppose $M$ satisfies \eqref{eq:bounded_geometry}, and $0<t\le T$ is such that $E_t\ne M$. Then $\p E_t$ is compact, and each connected component of $\p E_t$ has diameter bounded by $C\big(T,\Lambda,\cC(\p E_0)\big)$.

    (ii) If further $M$ has only one end, then $E_t\Subset M$ for $t$ as in (i).
\end{cor}
\begin{proof}
    (i) First, Lemma \ref{lemma:maximal_sol}(ia) provides a uniform gradient estimate in $M\setminus E_0$. Let $\Sigma$ be a connected component of $\p E_t$. Note that one can find at least $\lfloor\diam(\Sigma)/\Lambda\rfloor$ disjoint balls of the form $B(x_i,1/2\Lambda)$, with $x_i\in\Sigma$. If $\Sigma$ is noncompact, then we can find infinitely many such disjoint balls. Now note that:
    \begin{enumerate}[leftmargin=*, topsep=1pt, itemsep=1pt]
        \item If $d(x_i,E_0)\geq1/\Lambda$, then by Lemma \ref{lemma:reg_density}(i) we have $|\Sigma\cap B(x_i,1/2\Lambda)|\geq C(\Lambda,H_{\p E_0}^+)$.
        \item By volume comparison, among all the $x_i$, there are at most $C(\Lambda,\diam(\p E_0))$ many points such that $d(x_i,E_0)<1/\Lambda$.
    \end{enumerate}
    Furthermore, by Lemma \ref{lemma:maximal_sol}(id) we have $|\Sigma|\leq|\p E_t|\leq e^t|\p E_0|$.
    The diameter bound then follows from area comparison. Since we have proved that each component of $\p E_t$ has a uniform area lower bound, the compactness of $\p E_t$ again follows from area comparison.

    (ii) If $M$ has only one end, then Lemma \ref{lemma:maximal_sol}(ib)(ic) and the fact $E_t\ne M$ implies $E_t\Subset M$, due to Lemma \ref{lemma:one_end_cpt}.
\end{proof}

Finally, the following lemma shows that for $E_0$ a sufficiently small geodesic ball, there is a weak solution that remains proper for some definite amount of time. See the last picture in Figure \ref{f:bad_flows} for a weak solution that jumps to infinity at $t=0$.

\begin{lem}[no instant escape]\label{lemma:non_trivial} {\ }

    Assume $M$ satisfies \eqref{eq:bounded_geometry}, and fix $p\in M$. There is a sufficiently small $r_0>0$ such that: setting $E_0=B(p,r_0)$, the maximal solution $u$ given by Lemma \ref{lemma:maximal_sol}(i) satisfies $E_t(u)\Subset M$ for some $t>0$.
\end{lem}
\begin{proof}
    Note that \eqref{eq:bounded_geometry} implies a uniform lower bound on the volume of balls: there exists $V>0$, such that
    \[|B(x,1)|\geq V,\qquad\forall\,x\in M.\]
    This further implies $|M|=\infty$ since $M$ is connected and noncompact. Let us show that there exists constant $C(\Lambda)>0$ such that
    \begin{equation}\label{eq:non_trivial_aux1}
        |\p\Omega|\geq C(\Lambda)^{-1}\min\Big\{1,|\Omega|^{(n-1)/n}\Big\},\qquad\forall\,\Omega\Subset M\text{ with smooth boundary}.
    \end{equation}
    Once this is proved, then we recall the following result:
    \begin{theorem}[{\cite[Theorem 4.1]{Xu24proper}}]
        Let $M$ be of infinite volume, and whose isoperimetric profile function $\operatorname{ip}(v):=\inf\big\{|\p E|: E\Subset M, |E|=v\big\}$ satisfies
        \[\liminf_{v\to\infty}\operatorname{ip}(v)\geq A>0,\qquad \int_0^{v_0}\frac{dv}{\operatorname{ip}(v)}<\infty\quad\text{for some $v_0>0$.}\]
        Then for all $E_0$ with $|\p E_0|<A$, there exists a weak IMCF $u'$ in $M$ with initial value $E_0$, such that $E_t(u')\Subset M$ for some $t>0$.
    \end{theorem}
    Joining this result with \eqref{eq:non_trivial_aux1}, it follows that we may pick $r_0\ll1$ and obtain an IMCF $u'$ starting from $E_0$, such that $E_t(u')\Subset M$ for some $t>0$. Since $u$ is the maximal solution, it follows that $E_t(u)\Subset E_t(u')\Subset M$ for the same $t$. This proves the lemma.
    
    It remains to prove \eqref{eq:non_trivial_aux1}. We divide into two cases.

    \textbf{Case 1}: $|\Omega\cap B(x,1)|\geq V/2$ for some $x\in M$. By continuity, we can find another point $x'\in M$ such that $|\Omega\cap B(x',1)|=V/2$. Recall a well-known isoperimetric inequality:
    \begin{equation}\label{eq:isop_ineq}
        \min\big\{|B(x',1)\cap\Omega|,|B(x',1)\setminus\Omega|\big\}\leq C(\Lambda)\big|\p\Omega\cap B(x',1)\big|^{\frac{n}{n-1}},
    \end{equation}
    and note that the left hand side is at least $V/2$. Thus $|\p\Omega|\geq C(\Lambda)^{-1}$ for this case.

    \textbf{Case 2}: $|\Omega\cap B(x,1)|<V/2$ for all $x\in M$. Then \eqref{eq:isop_ineq} implies $|B(x,1)\cap\Omega|\leq C(\Lambda)|B(x,1)\cap\p\Omega|^{n/(n-1)}$ for all $x\in M$. By volume doubling, we may find finitely many balls $\{B(x_i,1)\}_{i=1}^m$ that covers $\overline\Omega$, whose covering multiplicity is bounded by $C(\Lambda)$.
    \[\begin{aligned}
        |\Omega|\leq\sum_{i=1}^m\big|\Omega\cap B(x_i,1)\big|
        &\leq C(\Lambda)\sum_{i=1}^m\big|\p\Omega\cap B(x_i,1)\big|^{\frac{n}{n-1}} \\
        &\leq C(\Lambda)\Big(\sum_{i=1}^m\big|\p\Omega\cap B(x_i,1)\big|\Big)^{\frac{n}{n-1}}
        \leq C(\Lambda)|\p\Omega|^{\frac n{n-1}}.
    \end{aligned}\]
    Then \eqref{eq:non_trivial_aux1} follows by combining the two cases.
\end{proof}

\begin{lem}\label{l: three cases}
    Suppose $(M,g)$ satisfies \eqref{eq:bounded_geometry} and has only one end. Let $E_0=B(p,r_0)$ be as in Lemma \ref{lemma:non_trivial}, and $u$ be the maximal IMCF with initial condition $E_0$, as given by Lemma \ref{lemma:maximal_sol}(i). Then $u$ is either proper, sweeping, or instantly escaping.
\end{lem}
\begin{proof}
    Let $T':=\sup\big\{t\geq0: E_t\Subset M\big\}$. By Lemma \ref{lemma:non_trivial} we have $T'>0$. If $T'=\infty$ then $u$ is proper. Now assume $T'<\infty$. By Corollary \ref{cor:diam_bound}(ii), we must have $E_t=M$ for all $t>T'$. Therefore $T'=\sup(u)$. Finally, consider $E_{T'}$: if $E_{T'}\Subset M$ then $u$ is instantly escaping by Definition \ref{def:sol_types}. Otherwise, we must have $E_{T'}=M$ by Corollary \ref{cor:diam_bound}(ii) again. Hence $M=\bigcup_{t<T'}E_t$ with $E_t\Subset M$ for each $t<T'$. This implies that $u$ is sweeping by Definition \ref{def:sol_types}.
\end{proof}

\begin{rmk}
    Both the bounded geometry and the one-endedness assumption cannot be removed. Indeed, Figure \ref{f:bad_flows} shows three examples where:
    
    (i) $M$ has bounded geometry, but the maximal IMCF partially rushes to infinity at time $T$ due to the presence of a cylindrical end;
    
    (ii) for a manifold with unbounded geometry, the maximal IMCF could partially rush to infinity in finite time;
    
    (iii) $M$ has an end with finite volume (thus has unbounded geometry), and the maximal IMCF rushes to infinity at time zero. 

    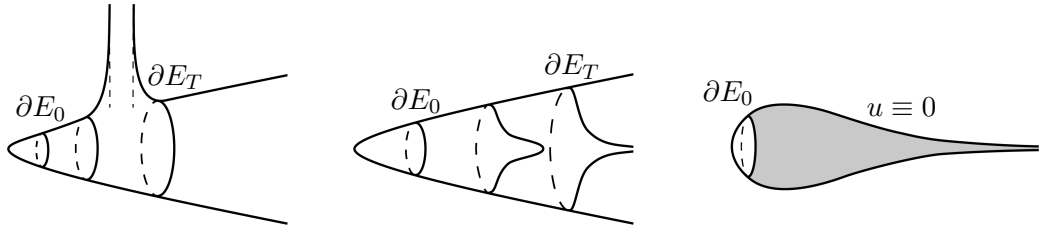
\begin{figure}[h]
    \centering
    \includegraphics[width=0.9\textwidth]{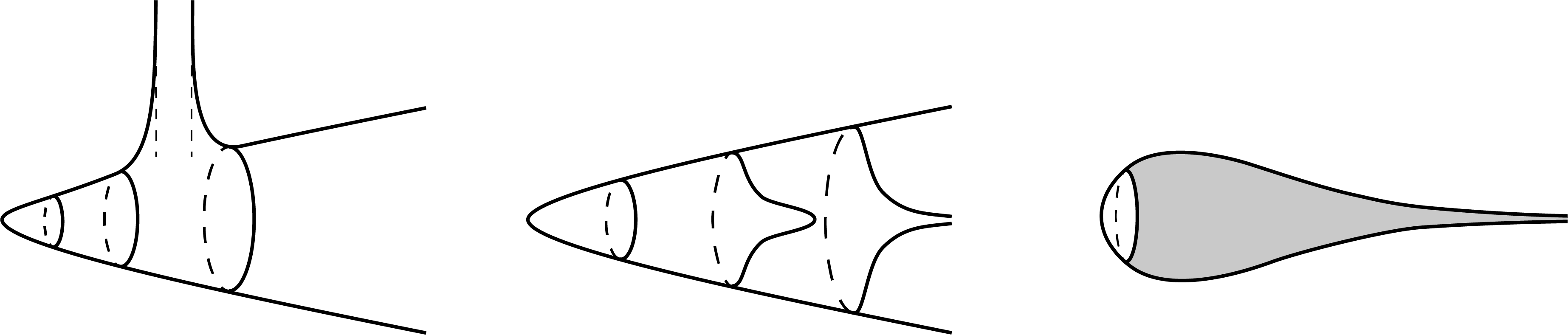}
    \caption{Examples of maximal IMCF that do not fall in Definition \ref{def:sol_types}.}\label{f:bad_flows}
    \begin{picture}(0,0)
        \put(-192,75){$\p E_0$}
        \put(-142,87){$\p E_T$}
        \put(-50,78){$\p E_0$}
        \put(8,91){$\p E_T$}
        \put(68,81){$\p E_0$}
        \put(130,75){$u\equiv0$}
    \end{picture}
\end{figure}
\end{rmk}

\section{Instantly escaping IMCF}\label{sec:escaping}

Consider an instantly escaping maximal solution $u$ of the IMCF with a bounded initial value (see Definition \ref{def:sol_types}). This means that there exists $T\in(0,\infty)$ such that
\[E_T(u)\Subset M,\qquad u\equiv T\ \ \text{in}\ \ M\setminus E_T(u).\]
The proof of the main theorems \ref{thm-main:R3}, \ref{thm-main:handlebody} are based on finding exhaustions of $M$ by level sets of IMCFs. However, for an instantly escaping solution $u$, there is no level set outside $E_T(u)$. We will modify the underlying metric and slightly enlarge it at infinity, so that a level set of the new maximal IMCF will appear in the edited region. Letting the edited region tend to infinity and the perturbation tend to zero, we finally obtain the desired sequence of exhausting surfaces. The main result of this section is Theorem \ref{c: almost minimizing exhaustion}, which is based on Lemma \ref{lemma:perturb} (the main perturbation lemma).

\begin{figure}[h]
    \vspace{18pt}
    \centering
    \includegraphics[width=0.9\linewidth]{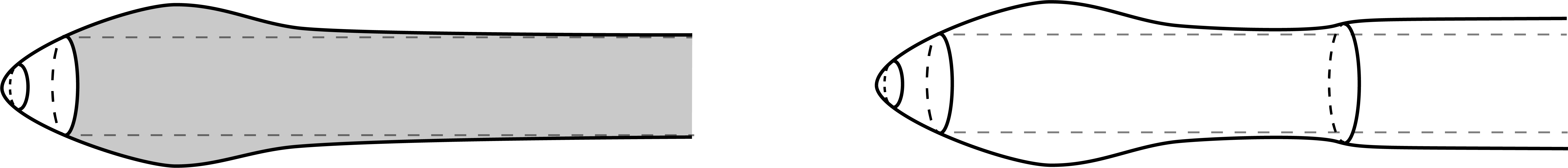}
    \begin{picture}(0,0)(198,35)
        \put(-150,52){$u\equiv T$}
        \put(-210,40){$\p E_0$}
        \put(-190,30){$\p E_T$}
        \put(10,40){$\p E_0$}
        \put(30,30){$\p E_T$}
        \put(116,28){$\p E_{T+t}(\tilde u)$}
        \put(20,80){$\overbrace{\hspace{105pt}}^{\tilde g=g}$}
        \put(145,80){$\overbrace{\hspace{45pt}}^{\tilde g=(1+\eps)g}$}
    \end{picture}
    \vspace{3pt}
    \caption{Left: the instantly escaping maixmal IMCF. Right: the metric perturbation and the new level set (see Lemma \ref{lemma:perturb} below).}\label{fig:perturb}
\end{figure}
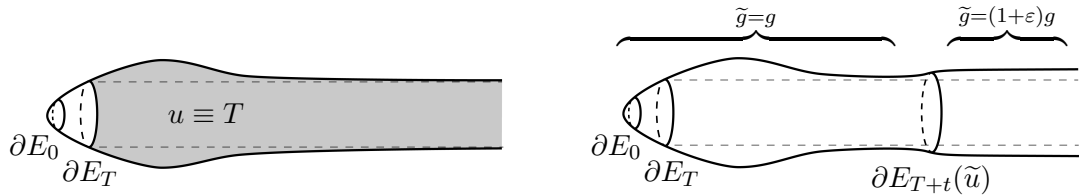

First, we prove the following lemma which says that $|\p E_T|$ equals the ``circumference at infinity'' of $M$.

\begin{lem}\label{c: a first exhaustion}
    Suppose $u$ is an escaping maximal IMCF starting from some $C^{1,1}$ initial value $E_0\Subset M$. Let $T$ be the escaping time of $u$. Then there exists an exhaustion of $M$ by $C^{1,\alpha}$ domains $\{F_k\}_{k=1}^\infty$ such that $E_T(u)\Subset F_1\Subset F_2\Subset\cdots\Subset M$ and $\lim_{k\to\infty}\Ps{F_k}=|\p E_T(u)|$.
\end{lem}
\begin{proof}
    First, we fix a sequence of precompact smooth domains $\{\Omega_l\}_{l=1}^\infty$ with $E_T(u)\Subset\Omega_1\Subset\Omega_2\Subset\cdots\Subset M$ and $\bigcup_{l\ge1}\Omega_l=M$.
    By Lemma \ref{lemma:maximal_sol}(ii), we know that there exists a sequence of weak solutions $u_l\in\Lip_{\loc}(\Omega_l)$ that converge to $u$ in $C^0_{\loc}$, and by Lemma \ref{lemma:maximal_sol}(iie), we have $E_T(u_l)=E_T(u)$ for all $l$. For each $\eps>0$, note the following facts:
    \begin{enumerate}[topsep=0ex, itemsep=0.3ex]
        \item $\bigcup_{l\geq1}E_{T+\eps}(u_l)=M$. This follows from $u_l\xrightarrow{C^0_{\loc}}u$ and $u\leq T$.
        \item By Lemma \ref{lemma:maximal_sol}(iid) and the above discussions, we have
        \[\P{E_{T+\eps}(u_l)}\leq e^\eps\P{E_T(u_l)}=e^\eps\P{E_T(u)}\qquad\text{for all $l$}.\]
    \end{enumerate}
    So we can inductively choose a sequence $l_k\to\infty$ so that $$E_{T+1/k}(u_{l_k})\Supset B_g(p,k)\cup E_{T+1/(k-1)}(u_{l_{k-1}}).$$ 
    Since $\p E_T(u)$ is outward minimizing, this implies $$|\p E_T(u)|\le |\p E_{T+1/k}(u_{l_k})|\le e^{1/k}|\p E_T(u)|.$$
    This proves the lemma by setting $F_k=E_{T+1/k}(u_{l_k})$.
\end{proof}

The following lemma shows that for any compact subset $K$, we can perturb the metric near infinity, so that the maximal solution ``jumps" out of $K$ but stays compact until a slightly larger time $t>T$.

\setcounter{claim}{0}

\begin{lem}\label{lemma:perturb}
    Suppose $(M^n,g)$ has only one end, and there are $\Lambda_k>0$ for each $k$ such that
    \begin{equation}\label{eq:higher_bounded_geometry}
        \inj\geq\Lambda^{-1}_0,\qquad  |\D^k\Rm|\leq\Lambda_k.
    \end{equation}
    Suppose $u$ is an instantly escaping maximal IMCF on $M$ with an initial value $E_0\Subset M$. Let $T\in(0,\infty)$ be the escape time of $u$.
    
    Then for any $K\Subset M$ and $\epsilon\in(0,1)$, there is a smooth metric $\tilde g$ with $\|\tilde g-g\|_{C^{10}_g}\leq\epsilon$ and with the following properties. If $\tilde u$ is the maximal IMCF in $(M,\tilde g)$ starting from $E_0$, then:
    
    (i) $E_T(\tilde u)=E_T(u)$,
    
    (ii) there exists $t\in(T,T+\eps)$ such that
    $E_t(\tilde u)\Subset M$ and $\p E_t(\tilde u)\cap(M\setminus K)\ne\emptyset$.
\end{lem}
\begin{proof}
Without loss of generality, we may assume $K\Supset E_T(u)$.
    By Lemma \ref{c: a first exhaustion}, we can replace $K$ by a larger domain and assume that it satisfies
    \begin{equation}\label{e: a small diam surface F_k}
        \Ps{K}_g<(1+\epsilon)^{n-2}|\p E_T(u)|.
    \end{equation}
    Fix a point $p\in E_0$. Assume $K\Subset B_g(p,R_1)$ for some $R_1>0$. We use $C>0$ to denote a generic constant that only depends on finitely many $\Lambda_k$.

    \begin{claim}\label{claim: perturbed metric}
        There exist $R_2>R_1$ and a smooth Riemannian metric $\tilde g$, such that
        \begin{enumerate}[label={(\roman*)}]
            \item $\tilde g\ge g$ and $\|\tilde g-g\|_{C^{10}_g}\leq C\epsilon$;
            \item $\tilde g=g \textnormal{ in }B_g(p,R_1)$, and $\tilde g=(1+\eps)g$ in $M\setminus B_g(p,R_2)$. 
        \end{enumerate}
    \end{claim}
    
    \begin{proof}
        First, we smooth the distance function $d_g(p,\cdot)$ by evolving it by a heat equation: Let $w_t$ be the solution to $\partial_t w_t=\Delta_{g} w_t$ with $w_0= d_g(p,\cdot)$. 
            Then $w_t$ is smooth for any $t>0$.
            For any fixed $x\in M$, we define a time-dependent function $v_t(y):=w_t(y)-w_0(x)$ for all $y\in M$, then by the triangle inequality we have $|v_0(y)|\le d_g(x,y)$.
            By the assumption \eqref{eq:higher_bounded_geometry}, the heat kernel satisfies the Gaussian bound $H(x,t;y,s)\le \tfrac{C}{(t-s)^{n/2}}e^{-\frac{d_g^2(x,y)}{C(t-s)}}$ for all $s<t\in [0,C^{-1}]$, see e.g. \cite[Corollary 26.26]{RFTandA3}. So using the representation formula
            \[w_t(x)=\int_M H(x,t;y,0)w_0(y)\,d\textnormal{vol}_gy,\]
            it is easy to see  $\|v_t\|_{C^0(B_g(x,1))}\le C$ for all $t\in[0,C^{-1}]$.
            So by the standard parabolic estimate \cite{Krylov1996}, it follows for each $t\in[\tfrac12C^{-1},C^{-1}]$ that
            \[\|w_t-w_0(x)\|_{C^{10}_g(B_g(x,1))}= \|v_t\|_{C^{10}_g(B_g(x,1))}\le C\|v_t\|_{C^0_g(B_g(x,1)\times[0,C^{-1}])}\le C.\]
            Therefore, the function $\tilde d:=w_{C^{-1}}$ satisfies
            \begin{equation*}\label{e: cut off function 2}
                \|\nabla \tilde d\|_{C^9_g}\leq C, \quad\textnormal{and}\quad|\tilde d(x)-d_g(p,x)|\le C \textnormal{ for all }x\in M.
        \end{equation*}
        Let $R_2>R_1+3C$ and let $\eta:\mathbb R\to[0,1]$ be a smooth non-decreasing function, such that $\eta\equiv0$ on $(-\infty,R_1+C]$, $\eta\equiv1$ on $[R_2-C,\infty)$, and $\|\eta\|_{C^{10}_g}\le 100$.  Then the metric $\tilde g=\big(1+\eps\eta\circ \tilde d\,\big)\cdot g$ satisfies 
        all assertions.    
    \end{proof}

    In the following, we fix the choice of $R_2>R_1$ and $\tilde g$ as in Claim \ref{claim: perturbed metric}. Then we see that $\tilde g$ satisfies \eqref{eq:higher_bounded_geometry} for all $k\le 8$ with possibly larger $\Lambda_k$ depending only on $C$. Thus Lemma \ref{lemma:reg_density} and Corollary \ref{cor:diam_bound} hold in $(M,\tilde g)$ with a weaker constant. It is helpful to recall the chain of inclusions
    \[E_T(u)\Subset K\Subset B_g(p,R_1)\Subset\big\{\tilde g=g\big\}\Subset\big\{\tilde g\ne(1+\eps)g\big\}\Subset B_g(p,R_2).\]

    \begin{claim}
        $u$ is also a weak solution of IMCF in $(M,\tilde g)$ with initial value $E_0$.
    \end{claim}
    \begin{proof}
        By Definition \ref{def:ivp}, it suffices to show that $u$ is a weak solution of IMCF in $(M\setminus\overline{E_0},\tilde g)$.
        By Definition \ref{def:weak_sol}, this is to show that 
        for any $t\in\R$, $K'\Subset M\setminus\bar{E_0}$ and set $F$ with $F\Delta E_t\Subset K'$, we have $J_{u,\tilde g}^{K'}(E_t)\leq J_{u,\tilde g}^{K'}(F)$. If $t>T$ then $E_t=M$, hence $F=M\setminus K''$ for some $K''\Subset M\setminus\bar{E_0}$. In this case $J_{u,\tilde g}^{K'}(E_t)\leq J_{u,\tilde g}^{K'}(F)$ is equivalent to $|\p K''|+\int_{K''}|\D_{\tilde g}u|\geq0$, which is obviously true.
        Now suppose $t\leq T$. Note the following facts: $\tilde g\geq g$, and $\p E_t(u)\subset\{\tilde g=g\}$, and $\int_A|\D_{\tilde g}u|\,dV_{\tilde g}=\int_A|\D_g u|\,dV_g$ for all set $A$ (since $u\equiv T$ wherever $\tilde g\ne g$). As a result, we have
        \[\begin{aligned}
            J_{u,\tilde g}^{K'}(F)
            &= |\p F\cap K'|_{\tilde g}-\int_{F\cap K'}|\D_{\tilde g}u|\,dV_{\tilde g}
            \geq |\p F\cap K'|_g-\int_{F\cap K'}|\D_g u|\,dV_g \\
            &= J_{u,g}^{K'}(F)
            \geq J_{u,g}^{K'}(E_t)
            = J_{u,\tilde g}^{K'}(E_t),
        \end{aligned}\]
        proving the claim.
    \end{proof}

    Let $\tilde u$ be the maximal weak IMCF in $(M,\tilde g)$ with initial value $E_0$. Hence we have $\tilde u\geq u$.  This implies $E_T(\tilde u)\subset E_T(u)\Subset M$. Since $u$ is also a weak solution in $(M,\tilde g)$ with initial value $E_0$, by Lemma \ref{lemma:max_principle} we have $E_T(\tilde u)=E_T(u)$.

    \begin{claim}
        For all $t<T+(n-1)\log(1+\epsilon)$, we have $E_t(\tilde u)\Subset M$.
    \end{claim}
    \begin{proof}
        Fix such a $t$. Let $\{\Omega_l\}_{l=1}^\infty$ be a smooth precompact exhaustion with
        \[B_g(p,R_2)\Subset\Omega_1\Subset\Omega_2\Subset\cdots\Subset M.\]
        Apply Lemma \ref{lemma:maximal_sol}(ii) to this exhaustion: we obtain a sequence $\tilde u_l$ of maximal weak IMCFs in $(\Omega_l,\tilde g)$ with the same initial value $E_0$. Then $\{\tilde u_l\}$ is a descending sequence, and $\tilde u_l\searrow\tilde u$ in $C^0_{\loc}(M\setminus E_0)$. In particular, note that $E_T(\tilde u_l)\subset E_T(\tilde u)\Subset\Omega_l$. Applying Lemma \ref{lemma:max_principle} in $\Omega_l$, it follows that $E_T(\tilde u_l)=E_T(\tilde u)=E_T(u)$ for each $l$.
    
        By Corollary \ref{cor:diam_bound}(ii) and \eqref{eq:higher_bounded_geometry} and the fact that $M$ is one-ended, to prove the claim it suffices to prove that $E_t(\tilde u)\ne M$. Suppose this is not the case. As $E_t(\tilde u)=\bigcup_{l\geq1}E_t(\tilde u_l)$, we would have $E_t(\tilde u_l)\Supset B_g(p,R_2)$ for some sufficiently large $l$.
        This implies
        \[\p E_t(\tilde u_l)\subset M\setminus B_g(p,R_2)\subset\big\{\tilde g=(1+\eps)g\big\}.\]
        So by our definition of $t$ and Lemma \ref{lemma:maximal_sol}(iid), we have
        \[|\partial E_t(\tilde u_l)|_g
            = \frac{|\partial E_t(\tilde u_l)|_{\tilde g}}{(1+\epsilon)^{n-1}}
            \le e^{t-T}\frac{|\p E_T(\tilde u_l)|_{\tilde g}}{(1+\eps)^{n-1}}
            = e^{t-T}\frac{|\p E_T(u)|_g}{(1+\eps)^{n-1}}
            < |\p E_T(u)|_g.\]
        So we find a surface outside of $\overline{E_T(u)}$ that has a strictly smaller $g$-perimeter, but this contradicts the outward minimizing of $E_T(u)$.
    \end{proof}

    \begin{claim}
        If $T+(n-2)\log(1+\epsilon)<t<T+(n-1)\log(1+\epsilon)$, then $ \overline{E_t(\tilde u)}\cap(M\setminus K)\ne\emptyset$.
    \end{claim}
    \begin{proof}
        Suppose the claim is false at such a time $t$. Note that $E_t(\tilde u)\subset K\Subset M$. By Lemma \ref{lemma:ivp_properties}(iii) and noting that $\tilde g=g$ in $K$ and $E_T(\tilde u)=E_T(u)\Subset K$, we have
        \[\P{E_t(\tilde u)}_{\tilde g}=e^{t-T}\P{E_T(\tilde u)}_{\tilde g}>(1+\eps)^{n-2}|\p E_T(u)|_g.\]
        But recalling \eqref{e: a small diam surface F_k} and $\tilde g=g$ on $K$ this implies
        \[\Ps{K}_{\tilde g}=\Ps{K}_g<(1+\epsilon)^{n-2}|\p E_T(u)|_g< \P{E_t(\tilde u)}_{\tilde g},\] 
        contradicting the outward minimizing property of $\p E_t(\tilde u)$ since we assumed $E_t(\tilde u)\subset K$.
    \end{proof}

    Combining Claim 2 and 3, the lemma is proven.
\end{proof}

\begin{theorem}\label{c: almost minimizing exhaustion}
    Suppose that $M^n$ is one-ended and satisfies $b_1(M)<\infty$, and that there are $\Lambda_k>0$ for each $k$ such that
    \begin{equation}\label{eq:higher_bounded_geometry_2}
        \inj\geq\Lambda^{-1}_0,\quad  |\D^k\Rm|\leq\Lambda_k.
    \end{equation}
    Suppose $u$ is an escaping maximal IMCF on $M$ with a $C^{1,1}$ initial value $E_0\Subset M$. Let $T\in(0,\infty)$ be the escape time of $u$.

    Then there exists $C>0$ such that for any compact subset $K\supset E_T(u)$ and $\delta>0$, there exists a domain $\Omega\Supset K$ such that
    \begin{enumerate}[itemsep=0.3ex, topsep=0.3ex, label={(\roman*)}]
        \item $\p\Omega$ is a connected $C^{1,\alpha}$ surface with $\textnormal{diam}(\Omega)\le C$ and $C^{1,\alpha}$ norm controlled by $C$;
        \item For any $F$ with $E_T(u)\Subset F\Subset M$, we have $|\partial \Omega|_g\leq(1+\delta)|\partial F|_g$.
    \end{enumerate}
    In particular, there exists a sequence of exhaustion by precompact subsets $\{\Omega_k\}_{k=1}^\infty$ satisfying (i) and (ii) for a sequence $\delta_k\to0$ as $k\to\infty$.
\end{theorem}

\begin{proof}
   Note that $M$ is one-ended. Thus for any $K'\Subset M$ and $\eps<1/100$, by Lemma \ref{lemma:perturb} we find a metric $\tilde g$ with $\|\tilde g-g\|_{C^{10}_g}\le\eps$ and a maximal IMCF $\tilde u$ in $(M^n,\tilde g)$, such that $E_{T}(u)=E_{T}(\tilde u)$, and for some $t\in(T,T+\eps)$ we have $E_t(\tilde u)\Subset M$, $\p E_t(\tilde u)\cap(M\setminus K')\ne\emptyset$.

    By Lemma \ref{lemma:maximal_sol}(ib,ic), $E_t(\tilde u)$ is connected, and $M\setminus E_t(\tilde u)$ has only unbounded connected components. Combining Corollary \ref{cor:diam_bound}(i) and $\|\tilde g-g\|_{C^{10}_g}\leq1/100$, we have that each connected component $\Sigma$ of $\p E_t(\tilde u)$ satisfies
    \begin{equation}\label{eq:diam1}
        \diam_g(\Sigma)\leq C\big(T+1,\Lambda,|\p E_0|,H_{\p E_0}^+\big)=:C_1.
    \end{equation}
    In particular, this bound is independent of $K'$. Now let $S$ be the compact set obtained in Lemma \ref{lem:connectedness2}, and choose $K'$ such that $N_g(S\cup K,C_1)\Subset K'\Subset M$, where $N_g(X,r)$ denotes the $r$-neighborhood of a subset $X$. Then it follows from $\p E_t(\tilde u)\cap(M\setminus K')\ne\emptyset$ that some connected component of $\p E_t(\tilde u)$ is disjoint from $S\cup K$. Inserting this into Lemma \ref{lem:connectedness2}, we see that $\p E_t(\tilde u)$ is connected. Hence by \eqref{eq:diam1} again we have $\p E_t(\tilde u)\cap K=\emptyset$, which implies $E_t(\tilde u)\Supset K$.
    
    Moreover, by Lemma \ref{lemma:reg_density}(ii) and $\|\tilde g-g\|_{C^{10}_g}\leq1/100$, $\p E_t(\tilde u)$ has bounded $C^{1,\alpha}$ norm with respect to $\tilde g$, hence with respect to $g$ as well. This proves (i) by setting $\Omega=E_t(\tilde u)$.
    
    Since $E_{T }(\tilde u)$ is outward minimizing in $(M,\tilde g)$, for any $F\Supset E_{T}(\tilde u)$ we have
    \[(1+\eps)^{n-1}|\partial F|_{g}
        \ge|\partial F|_{\tilde g}
        \ge |\p E_{T}(\tilde u)|_{\tilde g}
        \ge e^{-(t-T)}|\partial \Omega|_{\tilde g}
        \ge e^{-\eps}|\partial \Omega|_{g}.\]
    By taking $\eps$ sufficiently small, this implies assertion (ii) since $E_T(\tilde u)=E_T(u)$.
\end{proof}

\section{Proof of the main theorems}\label{sec:proof}

Now we can prove Theorem \ref{thm-main:R3} and \ref{thm-main:handlebody}.
Recall that in these two theorems, the manifolds have non-negative scalar curvature and bounded geometry. 

\begin{rmk}\label{rmk:higher_deriv}
    Let $(M,g)$ be a manifold as in Theorem \ref{thm-main:R3} or \ref{thm-main:handlebody}.
    Without loss of generality, we may assume that the scalar curvature is strictly positive, and there are constants $\Lambda_k>0$ for each $k\geq0$ such that
    \begin{equation}\label{e: can take smooth limit}
        \inj\geq\Lambda_0^{-1},\qquad |\nabla^k\Rm|\leq\Lambda_k.
    \end{equation}

    Indeed, thanks to the bounded geometry assumption, we can run a smooth complete Ricci flow $(M,g(t))$, $t\in[0,t_0]$ with $g(0)=g$ for some $t_0>0$, and $(M,g(t_0))$ satisfies \eqref{e: can take smooth limit} by Shi's derivative estimates \cite{Shi1987derivative1}.
    Moreover, if $R>0$ does not hold everywhere on $(M,g(t_0))$, then by the strong maximum principle, $(M,g(t))$ is flat for all $t\in[0,t_0]$.
    In the context of Theorem \ref{thm-main:R3}, this implies $M\cong\R^3$. In the context of Theorem \ref{thm-main:handlebody}, this leads to a contradiction.
    Therefore, replacing $g$ with $g(t_0)$, we may assume that \eqref{e: can take smooth limit} holds and $R>0$ strictly.
\end{rmk}

\begin{proof}[Proof of Theorem \ref{thm-main:R3}]
    By Remark \ref{rmk:higher_deriv}, we may assume $R>0$ and \eqref{e: can take smooth limit} holds.

    Fix $E_0=B(p,r_0)$ as in Lemma \ref{lemma:non_trivial}, and let $u$ be the maximal weak solution of IMCF starting from $E_0$.  Since $(M,g)$ is contractible, by Lemma \ref{lem:contract-ends} it has one end, and then Lemma \ref{l: three cases} implies that $u$ is either proper, sweeping, or instantly escaping. We prove the main theorem in three cases.

    \textbf{Case 1:}  Suppose $u$ is proper. Since $\p E_0$ is connected and $M$ is contractible, by Lemma \ref{lemma:maximal_sol}(ib,\,ic) and Lemma \ref{lem:connectedness} we see that all the $\p E_t$ ($t>0$) are connected. By Lemma \ref{lem:monotonicity} we have
    \[\int_0^t4\pi\chi(\partial E_t)\,dt\ge\int_{\partial E_t} H^2\,d\mu_t - \int_{\partial E_0} H^2\,d\mu_0\ge - \int_{\partial E_0} H^2\,d\mu_0,\quad\forall\,t>0.\]
    Thus there exists a sequence $t_i\to\infty$ such that $\p E_{t_i}$ is either an $\mathbb S^2$ or $\mathbb T^2$, for all $i$.
    Since $u\in\Lip_{\loc}(M)$, the sets $E_{t_i}$ must form an exhaustion of $M$.
    Thus we obtain an exhaustion of $M$ by precompact domains with boundaries either $\mathbb S^2$ or $\mathbb T^2$.
    If there are infinitely many $\mathbb S^2$, then Lemma \ref{lem:contractS2} implies that $M\cong R^3$.
    If there are infinitely many $\mathbb T^2$, Lemma \ref{lem:contractT2} implies that $M\cong\R^3$.
    
    \textbf{Case 2:} Suppose $u$ is sweeping. Then recall from Definition \ref{def:sol_types} that there exists $T\in(0,\infty)$ such that $E_t\Subset M$ for all $t<T$, and $E_T=\bigcup_{t\in[0,T)}E_t=M$.
    By Lemma \ref{lemma:maximal_sol}(ib,\,ic) and Lemma \ref{lem:connectedness}, each $\p E_t$ is connected. Then by Corollary \ref{cor:diam_bound}, for all $t\in[0,T)$ we have that $\p E_t$ is uniformly $C^{1,\alpha}$-bounded and has uniformly bounded diameters.
    Thus, we can select a sequence $t_i\nearrow T$ such that
    \begin{equation}\label{eq-R3:sparseness}
        \lim_{i\to\infty}d_g\big(\p E_{t_{i-1}},\p E_{t_i}\big)=\infty.
    \end{equation}
    Next, recall that each $E_{t_i}$ is outward area-minimizing. Hence for each $i$ and set $F$ with $E_{t_{i-1}}\subset F\Subset M$, we have
    \begin{equation}\label{eq-R3:almost_min}
        |\p F|\geq\big|\p E_{t_{i-1}}(u)\big|=e^{t_{i-1}-t_i}\big|\p E_{t_i}(u)\big|.
    \end{equation}

    Now fix basepoints $p_i\in\p E_{t_i}$, so $p_i\to\infty$ as $i\to\infty$. By \eqref{e: can take smooth limit}, we may pass to a subsequence and assume $(M,g,p_i)$ converge smoothly to a manifold $(M_\infty,g_\infty,p_\infty)$ with $R\ge0$, and $\partial E_{t_i}$ converge in the $C^{1,\beta}$-sense for some $\beta\in(0,\alpha)$ to a connected closed surface $\Sigma_\infty\subset M_\infty$. Combining \eqref{eq-R3:almost_min} and \eqref{eq-R3:sparseness} and the standard set replacing argument (see for example \cite[Theorem 21.14]{Maggi}), it follows that $\Sigma_\infty$ is locally area-minimizing. Since $g_\infty$ has nonnegative scalar curvature, it follows by Schoen-Yau \cite{SchoenYau1979} that $\Sigma_\infty$ must be $\mathbb S^2$ or $\mathbb T^2$. Hence there is an infinite subsequence of $\big\{\p E_{t_i}\big\}$ which either consists of $\mathbb S^2$ or consists of $\mathbb T^2$. This proves the theorem as in the end of Case 1.

    \textbf{Case 3: } Suppose $u$ is escaping. Then let $\{\Omega_k\}_{k=1}^\infty$ be the exhaustion of $C^{1,\alpha}$ domains as in Theorem \ref{c: almost minimizing exhaustion}. Fixing basepoints $p_k\in \p\Omega_k$, we have that up to a subsequence, $(M,g,p_k)$ smoothly converges to $(M_\infty,g_\infty,p_{\infty})$ with $R\ge0$, and $\p\Omega_k$ converge in $C^{1,\beta}$-sense to a connected closed surface $\Sigma_\infty\subset M_\infty$. By Theorem \ref{c: almost minimizing exhaustion}(ii) and set replacing, $\Sigma_\infty$ is in fact locally area-minimizing. Arguing as in Case 2, this proves the theorem.
\end{proof}

\begin{proof}[Proof of Theorem \ref{thm-main:handlebody}]
    Suppose $(M,g)$ is a solid handlebody, where $g$ is complete with $R\ge0$ and bounded geometry. Due to Remark \ref{rmk:higher_deriv}, we may assume that $R>0$ and \eqref{e: can take smooth limit} holds. Same as the proof of Theorem \ref{thm-main:R3}, consider a maximal weak IMCF $u$ starting from a sufficiently small geodesic ball. In either case, the same argument in Theorem \ref{thm-main:R3} implies that $M$ can be exhausted by bounded domains with $\mathbb S^2$ or $\mathbb T^2$ boundaries. By Lemma \ref{lemma:handlebody_exhaustion}, the genus of $M$ is at most 1.
\end{proof}

\bibliography{refs}
\bibliographystyle{abbrv}

\end{document}